\documentclass[10pt]{amsart}
\usepackage{enumerate,amsmath,amssymb,latexsym,
amsfonts, amsthm, amscd, MnSymbol}

\theoremstyle{plain}

\newtheorem{theorem}{Theorem}

\newtheorem*{theorem*}{Theorem}

\numberwithin{equation}{section}

\newcommand{\ra}{\rightarrow}

\addtocounter{theorem}{-1}

\begin{document}

\title {Classical theorems in the Implicational Propositional Calculus}

\date{}

\author[P.L. Robinson]{P.L. Robinson}

\address{Department of Mathematics \\ University of Florida \\ Gainesville FL 32611  USA }

\email[]{paulr@ufl.edu}

 \keywords{}

\begin{abstract}

For formulas of the Implicational Propositional Calculus (IPC) that are theorems of the classical Propositional Calculus (PC) we show that PC proofs  yield IPC proofs. As a consequence, completeness of PC yields completeness of IPC. 

\end{abstract}

\maketitle

\medbreak

\section{Theorem} 

\medbreak 

Consider the following natural question regarding the relationship between PC (the classical Propositional Calculus) and IPC (the Implicational Propositional Calculus): let $X$ and $Y$ be well-formed formulas of IPC and assume that there exists a deduction of $Y$ from $X$ within PC; does it follow that there exists a deduction of $Y$ from $X$ within IPC itself? 

\medbreak 

To be specific, we follow the approach to PC taken by Church [1]: the set of propositional variables is augmented by a propositional constant $\mathfrak{f}$ for falsity; the conditional $\supset$ is the only primitive connective and Modus Ponens is the only inference rule; and there are three axiom schemes, namely 
\medbreak 
$(\#1) \; A \supset (B \supset A)$ \par 
$(\#2) \; [A \supset (B \supset C)] \supset [(A \supset B) \supset (A \supset C)]$ \par
$(\sim \sim) \; \; [(A \supset \mathfrak{f}) \supset \mathfrak{f}] \supset A.$ \par
\medbreak 
\noindent
For IPC we follow Robbin [2] in modifying the foregoing specifications: the propositional constant $\mathfrak{f}$ is removed completely; and the double negation axiom scheme $(\sim \sim)$ is replaced by the Peirce axiom scheme 
\medbreak 
$(\mathbb{P}) \; \; \; [(A \supset B) \supset A] \supset A.$ \par 
\medbreak 
\noindent
We write $K$ for the set comprising all PC formulas and $L$ for the set comprising all IPC formulas; thus, $L \subset K$. We also use $K$ and $L$ to indicate the corresponding systems: thus, $X \vdash_K Y$ indicates the existence of a deduction of $Y$ from $X$ within PC while $\vdash_L Z$ indicates that $Z$ is a theorem of IPC. Lastly, we use $\equiv$ to indicate IPC equivalence: $X \equiv Y$ precisely when $X \vdash_L Y$ and $Y \vdash_L X$. 

\medbreak 

In these terms, our opening question is as follows: let $X$ and $Y$ be formulas in $L$; does $X \vdash_K Y$ imply  $X \vdash_L Y$? 

\medbreak 

Of course, the answer to this question is in the affirmative. One argument in support of this answer may be presented in three steps. Step (1): $X \vdash_K Y$ implies $\vdash_K X \supset Y$ by the Deduction Theorem (DT). Step (2): The PC theorem $X \supset Y$ is a tautology by the soundness of PC; so $X \supset Y$ is an IPC theorem by the completeness of IPC. Step (3): $\vdash_L X \supset Y$ implies $X \vdash_L Y$ by Modus Ponens (MP). 

\medbreak 

Our main purpose here is to present a theorem that facilitates an alternative argument and has an attractive by-product. Let $T(K)$ denote the set comprising all PC theorems and $T(L)$ the set comprising all IPC theorems. 

\medbreak 

\begin{theorem} \label{main}
$T(L) = L \cap T(K)$.
\end{theorem}

\medbreak 

Our proof of this theorem is given in the next section; in principle, when $Z \in L$ has a PC proof, we show that this may be converted into an IPC proof of $Z$. 

\medbreak 

An alternative argument supporting the affirmative answer to our opening question is now immediate: if $X \in L$ and $Y \in L$ then $X \supset Y \in L$ so that $\vdash_K X \supset Y$ implies $\vdash_L X \supset Y$ by Theorem \ref{main}; this replaces Step (2) of our earlier argument, while Step (1) and Step (3) remain unchanged. 

\medbreak 

As a pleasing by-product, we may add to a growing list of proofs that the Implicational Propositional Calculus is complete (see [3] and references therein) by deducing completeness of IPC from completeness of PC: indeed, if $Z \in L$ is a tautology then $Z \in T(K)$ by completeness of PC so that $Z \in T(L)$ by Theorem \ref{main}. 

\medbreak 

\section{Proof}

\medbreak 

We begin with some general remarks. The axiom schemes $(\#1)$ and $(\# 2)$ are common to PC and IPC; these axiom schemes suffice for the Deduction Theorem (DT) as a derived inference rule: in both systems, if $\Gamma$ is a set of formulas then $\Gamma \cup \{ A \} \vdash B$ implies $\Gamma \vdash A \supset B$. A routine consequence is Hypothetical Syllogism (HS) in both systems: $A \supset B, B \supset C \vdash A \supset C$. 

\medbreak 

Further, the Peirce axiom scheme $(\mathbb{P})$ of IPC is a theorem scheme in PC: we may verify this as follows. Let $A, B \in K$ be PC formulas. First, $\mathfrak{f} \vdash_K B$ because $\mathfrak{f} \vdash_K (B \supset \mathfrak{f}) \supset \mathfrak{f}$ by $(\#1)$ and $(B \supset \mathfrak{f}) \supset \mathfrak{f} \vdash_K B$ by $(\sim \sim)$. By MP and DT it follows that $A \supset \mathfrak{f} \vdash_K A \supset B$. Two successive applications of MP yield $(A \supset B) \supset A, A \supset \mathfrak{f} \vdash_K \mathfrak{f}$ whence DT yields $(A \supset B) \supset A \vdash_K (A \supset \mathfrak{f}) \supset \mathfrak{f}$. Finally, $(\sim \sim)$ yields $(A \supset B) \supset A \vdash_K A$ and DT yields $\vdash_K [(A \supset B) \supset A] \supset A$. 

\medbreak 

We now recall a technical device familiar from the theory of IPC. Let $Q \in L$ be a fixed IPC formula; a precise choice will be made later, but for the time being this formula is arbitrary. For each IPC formula $Z \in L$ we write $QZ := Q(Z) := Z \supset Q$; iteration produces $QQ Z = (Z \supset Q) \supset Q$. We shall require several properties of this device in our proof of Theorem \ref{main}. 

\begin{theorem} \label{eq}
If $Z \in L$ is an IPC formula and $Q \vdash_L Z$ then $QQ Z \vdash_L Z$. 
\end{theorem} 

\begin{proof} 
As DT yields $\vdash_L Q \supset Z$ and MP yields $(Z \supset Q) \supset Q, Z \supset Q \vdash_L Q$ it follows that $(Z \supset Q) \supset Q, Z \supset Q \vdash_L Z$ by MP again whence $(Z \supset Q) \supset Q  \vdash_L (Z \supset Q) \supset Z$ by DT again. As $\vdash_L [(Z \supset Q) \supset Z] \supset Z$ is an instance of $(\mathbb{P})$, a final application of MP yields $(Z \supset Q) \supset Q \vdash Z$.
\end{proof} 

\medbreak 

Note that $Z \vdash_L QQ Z$ in any case: $Z, Z \supset Q \vdash_L Q$ by MP so $Z \vdash_L (Z \supset Q) \supset Q$ by DT. 

\medbreak 

\begin{theorem} \label{QQ}
If $X \in L$ and $Y \in L$ are IPC formulas then $QQ (X \supset Y) \vdash_L QQ X \supset QQ Y$. 
\end{theorem} 

\begin{proof} 
Note that $QY, X \supset Y \vdash_L QX$ by HS whence $QQX, QY, X \supset Y \vdash Q$ by MP and therefore $QQ X, Q Y \vdash_L Q (X \supset Y)$ by DT. Thus $QQ (X \supset Y), QQX, Q Y \vdash_L Q$ by MP and so two applications of DT conclude the proof. 
\end{proof} 

\medbreak 

This deduction did not use the Peirce axiom scheme; the next deduction uses it. 

\begin{theorem} \label{supset}
If $X \in L$ and $Y \in L$ are IPC formulas then $ QQ X \supset QQ Y \vdash_L QQ (X \supset Y)$. 
\end{theorem} 

\begin{proof} 
On the one hand, $Q (X \supset Y) \vdash_L QQX$: indeed, $Q X, Q \supset Y \vdash_L X \supset Y$ by HS so that $Q (X \supset Y), Q X, Q \supset Y \vdash_L Q$ by MP and $Q (X \supset Y), Q X \vdash_L (Q \supset Y) \supset Q$ by DT while $[(Q \supset Y) \supset Q] \supset Q$ is an instance of $(\mathbb{P})$; now $Q (X \supset Y), Q X \vdash_L Q$ by MP and we apply DT. On the other hand, $Q (X \supset Y) \vdash_L Q Y$: in fact, $Y \vdash_L X \supset Y$ from $(\#1)$ so that $(X \supset Y) \supset Q, Y \vdash_L Q$ by MP; now apply DT. As $Q (X \supset Y)$ yields $QQ X$ and $Q Y$ it follows that $QQ X \supset QQY, Q (X \supset Y) \vdash_L Q$ by MP twice and then DT concludes the proof. 
\end{proof} 

\medbreak 

By definition, the degree of a formula is the number of conditionals in its formation. Now, when $Z \in K$ we define $\phi (Z) \in L$ by induction on degree as follows. If $Z$ has degree zero then $Z$ is either the constant $\mathfrak{f}$ or a variable $p$; we define $\phi(\mathfrak{f}) = Q$ and $\phi (p) = QQ p$. If $Z$ has positive degree then $Z = X \supset Y$ for unique $X$ and $Y$ in $K$ of lesser degree; we define $\phi (Z) = \phi (X) \supset \phi (Y)$. This map $\phi : K \ra L$ has a number of properties, among which we draw attention to two. 

\medbreak

\begin{theorem} \label{phiK}
If $Z \in K$ then $Q \vdash_L \phi (Z)$. 
\end{theorem} 

\begin{proof} 
By induction on degree, of course. The case of zero degree is clear: $Q \vdash_L Q$ is immediate, while the deduction $Q \vdash_L Qp \supset Q$ follows from an instance of $(\# 1)$ by MP. For the case of positive degree, note that if $Q \vdash_L \phi (Y)$ then $Q \vdash_L \phi (X) \supset \phi (Y)$ results from the instance $\vdash_L \phi (Y) \supset (\phi (X) \supset \phi (Y))$ of $(\# 1)$ by MP. 
\end{proof} 

\medbreak 

Had we simply defined $\phi (p) = p$ when $p$ is a propositional variable, this claim would fail. 

\medbreak 

\begin{theorem} \label{phiL}
If $Z \in L$ then $\phi (Z) \equiv QQ Z$. 
\end{theorem} 

\begin{proof} 
An IPC formula of zero degree is a variable $p$, for which $\phi (p) = QQ p$ by definition. Consider an IPC formula $Z = X \supset Y$ of positive degree: inductively, $\phi (X) \equiv QQ X$ and $\phi (Y) \equiv QQ Y$; Theorem \ref{QQ} and Theorem \ref{supset} now yield 
$$\phi (Z) := \phi (X) \supset \phi (Y) \equiv QQ X \supset QQ Y \equiv QQ (X \supset Y) = QQ Z.$$
\end{proof} 

\medbreak 

Incidentally, recall [1] that the negation of the PC formula $Z \in K$ is given by $\sim Z := Z \supset \mathfrak{f}$. The map $\phi$ brings out the r\^ole of $Q( \cdot )$ as a partial negation: if $Z \in K$ then 
$$\phi (\sim Z) = \phi (Z \supset \mathfrak{f}) = \phi (Z) \supset \phi(\mathfrak{f}) = \phi(Z) \supset Q = Q \phi(Z).$$

\medbreak 

We are now prepared to prove our main theorem, which we restate for convenience. 

\medbreak 

\medbreak

\noindent
{\bf Theorem 0.} $T(L) = L \cap T(K)$. 

\begin{proof} 
The inclusion $T(L) \subseteq L \cap T(K)$ is clear: a proof in IPC yields a proof in PC once each application of the axiom scheme $(\mathbb{P})$ of IPC is replaced by the (proof of the) corresponding theorem scheme in PC. 

\medbreak 

Now let $Z \in L \cap T(K)$: say the sequence $Z_0, Z_1, \dots , Z_N = Z$ of PC formulas constitutes a PC proof. Consider the following sequence $\mathcal{S}$ of IPC formulas: 
$$\phi (Z_0), \phi (Z_1), \dots , \phi (Z_N) = \phi (Z).$$
If $Z_n = X \supset (Y \supset X)$ is an instance of scheme $(\#1)$ in $K$ then $\phi (Z_n) = \phi (X) \supset (\phi (Y) \supset \phi (X))$ is an instance of $(\#1)$ in $L$; likewise, if $Z_n$ is an instance of scheme $(\#2)$ in $K$ then $\phi (Z_n)$ is an instance of $(\#2)$ in $L$. If $Z_n = [(W \supset \mathfrak{f}) \supset \mathfrak{f}] \supset W$ is an instance of the double negation scheme $(\sim \sim)$ then $\phi (Z_n) = QQ \phi (W) \supset \phi (W)$ is a theorem of IPC on account of Theorem \ref{eq} and Theorem \ref{phiK}. If $Z_n$ follows from earlier terms $Z_m$ and $Z_{\ell} = Z_m \supset Z_n$ by MP then $\phi (Z_n)$ follows from $\phi(Z_m)$ and $\phi (Z_{\ell}) = \phi (Z_m) \supset \phi (Z_n)$ by MP. Thus the sequence $\mathcal{S}$ furnishes an IPC proof of $\phi (Z)$. Half of Theorem \ref{phiL} informs us that $\phi (Z) \vdash_L QQ Z$; appending this derivation we deduce $\vdash_L QQ Z = (Z \supset Q) \supset Q$. 

\medbreak 

Up until this point, $Q \in L$ has been arbitrary; we now take $Q$ to be $Z$ itself, in which case MP combines the foregoing deduction $\vdash_L (Z \supset Z) \supset Z$ with the instance $\vdash_L [(Z \supset Z) \supset Z] \supset Z$ of $(\mathbb{P})$ to yield $\vdash_L Z$. 
\end{proof} 

\bigbreak

\begin{center} 
{\small R}{\footnotesize EFERENCES}
\end{center} 
\medbreak 

[1]  A. Church, {\it Introduction to Mathematical Logic}, Princeton University Press (1956). 

[2] J. W. Robbin, {\it Mathematical Logic - A First Course}, W.A. Benjamin (1969); Dover Publications (2006).

[3] P. L. Robinson, {\it $Q$-tableaux for Implicational Propositional Calculus}, arXiv 1512.03525 (2015). 

\medbreak

\end{document}